\documentclass[12pt, a4paper]{amsart}

\usepackage[hmargin=30mm, vmargin=25mm, includefoot, twoside]{geometry}
\usepackage[
  bookmarksopen=true,
  hypertexnames=false]{hyperref}

\usepackage{amsfonts,amssymb,verbatim}

\usepackage{latexsym}
\usepackage{mathrsfs}
\usepackage{stmaryrd}
\usepackage{xspace}
\usepackage{enumerate, paralist}
\usepackage{graphicx}
\usepackage[all]{xy}
\usepackage{extarrows}
\usepackage{tabu}
\usepackage{accents}
\usepackage{tensor}
\usepackage{amssymb}

\usepackage[usenames,dvipsnames]{color}

\usepackage{pxfonts}

\usepackage{amsmath}

\usepackage[tbtags]{mathtools}

\usepackage{amsthm, thmtools}

\newtheorem{thm}{Theorem}[section]

 \newtheorem{cor}[thm]{Corollary}
 \newtheorem{lem}[thm]{Lemma}
 \newtheorem{prop}[thm]{Proposition}

\newtheorem{introthm}{Theorem}

\numberwithin{equation}{section}

 \theoremstyle{definition}
  \newtheorem{defn}[thm]{Definition}

 \theoremstyle{remark}
 \newtheorem{rem}[thm]{Remark}
  \newtheorem{ex}[thm]{Example}
   
\newtheorem*{claim*}{Claim}

\def\NN{\mathbb{N}}
\def\CC{\mathbb{C}}
\def\RR{\mathbb{R}}
\def\ZZ{\mathbb{Z}}
\def\FF{\mathbb{F}}

\def\HS{\mathrm{HS}}
\def\I{\mathrm{I}}
\def\Id{\mathrm{Id}}
\def\GL{\mathrm{GL}}
\def\rank{\mathrm{rank}}
\def\Ad{\mathrm{Ad}}

\def\E{\mathcal E}
\def\H{\mathcal{H}}
\def\M{\mathcal M}
\def\G{\mathcal G}
\newcommand{\maN}{\mathcal N}

\newcommand{\SL}{\operatorname{SL}}
\newcommand{\tr}{\operatorname{Tr}}

\def\U{\mathfrak U}
\def\B{\mathfrak B}

\newcommand{\boxspace}{\mathord{\scalebox{1.4}{$\square$}}}

\begin{document}

\title{A-T-menability and metric amenability for HLS groupoids}

\author[Yu Qiao]{Yu Qiao}
\address{School of Mathematics and Statistics, Shaanxi Normal University,
Xi'an, 710119, Shaanxi, China}
\email{yqiao@snnu.edu.cn}

\author[Hao Song]{Hao Song}
\address{School of Mathematics and Statistics, Shaanxi Normal University,
Xi'an, 710119, Shaanxi, China}
\email{haosong202603@163.com}

\author[Yan Wang]{Yan Wang}
\address{School of Mathematics and Statistics, Shaanxi Normal University,
Xi'an, 710119, Shaanxi, China}
\email{20241624@snnu.edu.cn}

\author[Jiawen Zhang]{Jiawen Zhang}
\address{School of Mathematical Sciences, Fudan University,
Shanghai 200433, China}
\email{jiawenzhang@fudan.edu.cn}

\thanks{Jiawen Zhang was partly supported by the National Key R{\&}D Program of China 2022YFA100700 and NSFC 12422107.}

\thanks{Keywords: HLS groupoids; a-T-menability; box spaces; coarse embeddability into Hilbert space; metric amenability.}

\thanks{2020 Mathematics Subject Classification: Primary 22A22; Secondary 22D55, 51F30}

\begin{abstract}
We characterise coarse embeddability of box spaces into Hilbert space in terms of a-T-menability of the associated Higson–Lafforgue–Skandalis (HLS) groupoids. Moreover, we prove that a-T-menability and metric amenability are independent for HLS groupoids: all four possible combinations occur among HLS groupoids associated with approximating sequences of the free group $\mathbb F_2$. The two separating examples are constructed using relative expansion and twisted spectral-gap techniques, respectively.
\end{abstract}

\date{\today}

\maketitle

\section{Introduction}

A natural way to study analytic properties of a residually finite group is to approximate it by finite quotients and investigate their coarse geometry through the associated box space. More precisely, let $\Gamma$ be a finitely generated residually finite group and $\mathcal N=\{N_n\}_{n\in\NN}$ be a decreasing sequence of finite-index normal subgroups with trivial intersection, called an \emph{approximating sequence} of $\Gamma$. The associated box space $\boxspace_{\maN}(\Gamma)$ is the coarse disjoint union of the finite quotients $\Gamma/N_n$ with the quotient word metrics (see Definition~\ref{defn:box space}). A fundamental correspondence by Guentner (see \cite[Proposition~11.39]{Roe03_book} for a proof) is that $\Gamma$ is amenable if and only if $\boxspace_{\maN}(\Gamma)$ has property A, introduced in \cite{Yu00}. In particular, whether a box space has property A is independent of the chosen approximating sequence.

An important weakening of amenability is a-T-menability, which was introduced by Gromov \cite{Gromov88, Gromov93}. This property is also known as the Haagerup property, originating in Haagerup's pioneering work on free groups \cite{Haagerup78}. All amenable groups and non-abelian free groups are a-T-menable, and this property has significant applications in geometric group theory and operator algebras \cite{CCJJV}. For example, Higson and Kasparov proved the Baum-Connes conjecture with coefficients for a-T-menable groups \cite{HK01}.

It was shown that coarse embeddability of a box space into Hilbert space implies a-T-menability of the underlying group (see \cite[Proposition~11.26]{Roe03_book}), but the converse fails. Indeed, the free group $\mathbb F_2$ admits both coarsely embeddable box spaces \cite{AGS12} and congruence box spaces forming expander graphs \cite{LA94}. A precise correspondence between a-T-menability and the coarse geometry of box spaces was established in \cite{CWW13}: a finitely generated residually finite group is a-T-menable if and only if one (equivalently, every) box space of the group admits a fibred coarse embedding into Hilbert space, introduced in \cite{CWY13}. However, the ordinary coarse embeddability of a box space depends on the chosen approximating sequence.

The groupoid approach provides another perspective on the relationship between a-T-menability and coarse geometry. Following Tu's extension of a-T-menability from groups to locally compact groupoids \cite[Section 3]{Tu99} using continuous proper functions of negative type, Skandalis, Tu and Yu showed in \cite{STY02} that a uniformly discrete metric space $X$ of bounded geometry coarsely embeds into Hilbert space if and only if the coarse groupoid $G(X)$ is a-T-menable. In a related direction, Finn-Sell proved in \cite{Finn14} that fibred coarse embeddability of $X$ into Hilbert space implies a-T-menability of the boundary coarse groupoid $G(X)|_{\partial\beta X}$.

In the case of box spaces, there is a simpler groupoid arising directly from finite quotients, known as the Higson–Lafforgue–Skandalis (HLS) groupoid. The HLS groupoid was introduced in \cite{HLS02} to provide counterexamples to the Baum-Connes conjecture for groupoids. More precisely, for a finitely generated residually finite group $\Gamma$ and an approximating sequence $\mathcal N=\{N_n\}_{n\in\NN}$, the associated HLS groupoid $G_{\mathcal N}(\Gamma)$ is a group bundle, having the finite quotients $\Gamma/N_n$ as its fibres over $\mathbb N$ and $\Gamma$ as its fibre at infinity. See Section~\ref{sec:pre} for a precise definition.

In this paper, we investigate the relationship between a-T-menability of HLS groupoids and coarse embeddability of box spaces. Our first main result is the following (see Theorem~\ref{thm:CE iff aTmenable} for the full  statement).

\begin{introthm}\label{introthm:CE iff aTm}
Let $\maN=\{N_n\}_{n\in \NN}$ be an approximating sequence of a countable discrete group $\Gamma$. Then the box space $\boxspace_{\maN}(\Gamma)$ is coarsely embeddable into Hilbert space \emph{if and only if} the associated HLS groupoid $G_{\maN}(\Gamma)$ is a-T-menable.
\end{introthm}


The proof of Theorem~\ref{introthm:CE iff aTm} follows Roe's approach from \cite[Proposition~11.26]{Roe03_book}. By averaging negative-type kernels over the finite quotients and passing to a subsequence, one obtains a proper negative-type function on the group. Then the main difficulty is to ensure continuity at the fibre at infinity along the entire approximating sequence, rather than merely a subsequence. To overcome the issue, we develop a filling technique  (see Lemma~\ref{lem:filling lemma}) that handles the omitted fibres while preserving continuity and uniform properness.




Besides the coarse-geometric characterisation in Theorem~\ref{introthm:CE iff aTm}, we also consider a representation-theoretic property of HLS groupoids concerning the relationship between their maximal and reduced $C^*$-algebras. By Hulanicki's theorem (see \cite[Corollary~G.3.9]{BHVBook}), a discrete group is amenable if and only if its maximal and reduced group $C^*$-algebras coincide. For étale groupoids, the analogous condition is called \emph{metric amenability} (or the \emph{weak containment property}) \cite{AD16, SW13}. Topological amenability implies metric amenability \cite{ADR00}, but the converse fails even for HLS groupoids, as shown by Willett \cite{Wil15}. Moreover, Willett characterised metric amenability of HLS groupoids using the regular representations of finite quotients \cite[Lemma 2.7]{Wil15}. 

Hence, it is natural to ask whether a-T-menability and metric amenability of HLS groupoids imply one another. Our second main result gives a complete answer, even when the underlying group is fixed (see Table~\ref{tab:four-cases} for a summary).

\begin{introthm}\label{introthm:Examples}
There exist four approximating sequences of the free group \(\mathbb F_2\) such that the associated HLS groupoids realise all four possible combinations of a-T-menability and metric amenability.
\end{introthm}

Our constructions make use of Kazhdan's property (T) and property $(\tau)$ \cite{BHVBook, LA94}. The former provides a uniform spectral gap away from invariant vectors in all unitary representations, while the latter restricts to those factoring through a prescribed family of finite quotients. Based on Theorem~\ref{introthm:CE iff aTm}, we show that for groups without property (T), having property $(\tau)$ obstructs both a-T-menability and metric amenability of the associated HLS groupoid (see Corollary~\ref{cor:tau obstruction}). Thus, ordinary spectral gap cannot separate these two properties over $\mathbb F_2$. On the other hand, combining Willett's construction with the homology-cover method of Arzhantseva, Guentner and \v{S}pakula \cite{AGS12} (further developed by Khukhro \cite{Khu14}), we obtain an HLS groupoid possessing both properties (see Example~\ref{ex:both properties}).

To obtain metric amenability without a-T-menability, we explore the idea of relative expansion from \cite{AT15}. More precisely, we refine the approximating sequence in Willett's construction \cite[Lemma~2.8]{Wil15} by choosing finite quotients of $\mathrm{SL}_3(\mathbb Z)$ in a precise way. Combining property (T) of $\mathrm{SL}_3(\mathbb Z)$ with lower bounds on the dimensions of non-trivial irreducible representations of $\mathrm{SL}_3(\mathbb F_p)$, we establish a uniform spectral gap relative to suitable subgroups of the finite quotients (Corollary~\ref{cor:spetral gap}). This obstructs a-T-menability, while metric amenability is preserved under refinement (Theorem~\ref{thm:counterexample}). Moreover, this box space contains no weakly embedded expanders (Remark~\ref{rem:weakly embedded}).

For the reverse implication, we start with a congruence approximation of $\mathbb F_2$ having property $(\tau)$ and pass to its square refinement. Although this restores coarse embeddability by Proposition~\ref{lem:square refinement}, the original property \((\tau)\) obstruction no longer applies. Instead, we construct a character of order three and prove a twisted spectral-gap estimate for elementary abelian $2$-group extensions (Lemma~\ref{lem:twisted lifting}). Combined with Selberg's congruence spectral gap \cite{LA94}, this produces a self-adjoint element in the group algebra that is uniformly positive under the regular representations of finite quotients in the refined sequence but vanishes under the chosen character, contradicting metric amenability (Theorem~\ref{thm:reverse counterexample}).

\subsection*{Organisation} In Section~\ref{sec:pre}, we recall basic notions. In Section~\ref{sec:aTm and CE}, we establish Theorem~\ref{introthm:CE iff aTm} and discuss applications of the homology-cover construction. In Section~\ref{sec:Examples}, we investigate the roles of property (T) and property $(\tau)$ and construct HLS groupoids separating a-T-menability from metric amenability. Finally, Section~\ref{sec:conclusions} summarises the four possible combinations in Theorem~\ref{introthm:Examples}.

\subsection*{Acknowledgements}
Y.Q. and J.Z. would like to express their sincerest gratitude to Rufus Willett for enlightening discussions, email communications, and his generous sharing of his ideas. This article began with a discussion with him about a-T-menability of HLS groupoids and he kindly provided Example~\ref{ex:SL non+non}. J.Z would also like to thank Nick Wright for several enlightening discussions and email communications.

\subsection*{AI use statement} We acknowledge the use of OpenAI's ChatGPT in exploring mathematical ideas related to some of the results in this paper and in improving parts of the exposition. All mathematical arguments were carefully checked and revised by the authors, who bear full responsibility for the final manuscript.

\section{Preliminaries}\label{sec:pre}

We collect some basic notions needed for this paper.

\begin{defn}[{\cite[Definition 2.1]{Wil15}}]\label{defn:app seq}
An {\em approximating sequence} for a countable discrete group $\Gamma$ is a family $\maN=\{N_n\}_{n\in \NN}$
of decreasing normal subgroups of $\Gamma$ with finite index such that $\bigcap_n N_n$ is trivial.
An {\em approximated group} is a pair $(\Gamma, \maN)$ consisting of a countable discrete group $\Gamma$ together with an approximating sequence $\maN$. We say that $\Gamma$ is \emph{residually finite} if it admits an approximating sequence.
\end{defn}

Let $\maN=\{N_n\}_{n\in \NN}$ be an approximating sequence of a countable discrete group $\Gamma$. For each $n\in \NN$, set $\Gamma_n:= \Gamma / N_n$ and $\Gamma_\infty:=\Gamma$. Also denote $\pi_n: \Gamma \to \Gamma_n$ the canonical quotient map for $n\in \NN$ and $\pi_\infty: \Gamma \to \Gamma_\infty$ the identity map. 

Now we recall the HLS groupoid constructed in \cite{HLS02} (see also \cite{Wil15}) as follows.

\begin{defn}[{\cite[Definition 2.2]{Wil15}}]\label{defn:HLS groupoid}
Let $\maN=\{N_n\}_{n\in \NN}$ be an approximating sequence of a countable discrete group $\Gamma$. Define $G_{\maN}(\Gamma)$ to be the group bundle
\[
G_{\maN}(\Gamma):= \bigsqcup_{n\in \NN \cup \{\infty\}} \, \{n\} \times \Gamma_n.
\]
This forms a groupoid, whose source and range maps are the projection onto the first coordinate. Equip $G_{\maN}(\Gamma)$ with the topology generated by the following sets:
\begin{enumerate}
 \item For each $n\in \NN$ and $g\in \Gamma_n$, the singleton $\{(n,g)\}$ is open;
 \item For each $g\in \Gamma$ and $N \in \NN$, $\{(n,\pi_n(g)): n\in \NN \cup \{\infty\} \text{ and } n>N\}$ is open.
\end{enumerate}
The associated \emph{Higson-Lafforgue-Skandalis (HLS) groupoid} is the groupoid $G_{\maN}(\Gamma)$ with the topology above. Clearly, $G_{\maN}(\Gamma)$ is a locally compact Hausdorff second countable \'{e}tale groupoid with unit space the one-point compactification of $\NN$. 
\end{defn}

Let us recall another geometric object associated to an approximated group.

\begin{defn}\label{defn:box space}
Let $\maN=\{N_n\}_{n\in \NN}$ be an approximating sequence of a countable discrete group $\Gamma$ equipped with a proper length function. For each $n\in \NN$, endow the quotient $\Gamma_n$ with the quotient length metric.  The associated \emph{box space} $\boxspace_{\maN}(\Gamma)$ is the coarse disjoint union of $\Gamma_n$, \emph{i.e.}, as a set 
\[
\boxspace_{\maN}(\Gamma) := \bigsqcup_{n\in \NN} \Gamma_n,
\]
and the metric $d$ on $\boxspace_{\maN}(\Gamma)$ satisfies $d(\Gamma_n, \Gamma_m) \to +\infty$ as $n+m \to +\infty$ for $n\neq m$. 
\end{defn}

It is routine to check that the coarse equivalence class of the box space is independent of the choice of proper length function on \(\Gamma\) and of the metric used to form the coarse disjoint union. Hence up to coarse equivalence, the notion of box space is well-defined for any approximated countable discrete group.

Recall that a metric space $(X,d)$ is \emph{coarsely embeddable into Hilbert space} if there exist a Hilbert space $\H$, a map $f:X \to \H$ and two non-decreasing functions $\rho_-, \rho_+: [0,+\infty) \to [0,+\infty)$ with $\lim\limits_{t\to +\infty} \rho_{\pm}(t) = +\infty$ such that
\[
\rho_-(d(x,y)) \leq \|f(x) - f(y)\|_{\H} \leq \rho_+(d(x,y)) \quad \text{for any} \quad x,y\in X.
\]

For a locally compact Hausdorff \'{e}tale groupoid $\G$, the space $C_c(\G)$ of all compactly supported continuous functions on $\G$ forms a $\ast$-algebra by convolution. The \emph{maximal groupoid $C^*$-algebra} $C^*_{\max}(\G)$ is the completion of $C_c(\G)$ over all unitary representations of $\G$, and the \emph{reduced groupoid $C^*$-algebra} $C^*_r(\G)$ is the completion over the left regular representation. Say that $\G$ is \emph{metrically amenable} \cite{SW13} (or $\G$ has the \emph{weak containment property} \cite{AD16}) if the canonical quotient $C^*_{\max}(\G)\to C^*_r(\G)$ is an isomorphism. For a group $\Gamma$, denote by $\lambda_\Gamma$ its left regular representation.

We record the following characterisation by Willett (see \cite[Lemma~2.7]{Wil15}).

\begin{lem}\label{lem:willett criterion}
Let $\maN=\{N_n\}_{n}$ be an approximating sequence of a countable discrete group $\Gamma$. Then $G_{\maN}(\Gamma)$ is metrically amenable \emph{if and only if} the following holds:
\begin{equation}\label{EQ:WCP char}
\|x\|_{C^*_{\max}(\Gamma)} =\sup_{n\in\NN}\|\tilde{\lambda}_{\Gamma/N_n}(x)\| \quad \text{for any} \quad x\in \CC[\Gamma].
\end{equation}
Here $\tilde{\lambda}_{\Gamma/N_n}$ is the left regular representation pulled back to $\Gamma$ and $\CC[\Gamma]$ is the group ring. 
\end{lem}

Unlike the formulation in \cite[Lemma~2.7]{Wil15}, the supremum in \eqref{EQ:WCP char} does not include the left regular representation \(\lambda_\Gamma\). Indeed, for any \(x\in\CC[\Gamma]\) and finitely supported \(\xi\in\ell^2(\Gamma)\), the quotient maps are eventually injective on the union of the supports of \(\xi\) and \(\lambda_\Gamma(x)\xi\). Hence,
\[
\|\lambda_\Gamma(x)\xi\| \le \sup_{n\in \NN}\|\tilde{\lambda}_{\Gamma/N_n}(x)\| \cdot\|\xi\|,
\]
which gives the desired inequality by density.

The following tool is useful to deduce metric amenability. It shows that metric amenability is preserved under refinement
of approximating sequences.

\begin{cor}\label{cor:refinement preserves WCP}
Let $\mathcal N=\{N_n\}_n$ and $\mathcal M=\{M_n\}_n$ be two approximating sequences of a countable discrete group $\Gamma$. Suppose that for every $n$, there exists $m$ such that $M_m\subseteq N_n$. If $G_{\mathcal N}(\Gamma)$ is metrically amenable, then so is $G_{\mathcal M}(\Gamma)$.
\end{cor}

\begin{proof}
Whenever $M_m\subseteq N_n$, the representation $\tilde{\lambda}_{\Gamma/N_n}$ is a subrepresentation of
$\tilde{\lambda}_{\Gamma/M_m}$. Hence, for any $x\in\CC[\Gamma]$, Lemma~\ref{lem:willett criterion} implies 
\[
\|x\|_{C^*_{\max}(\Gamma)}
=\sup_{n\in \NN}\|\tilde{\lambda}_{\Gamma/N_n}(x)\| \leq\sup_{m\in \NN}\|\tilde{\lambda}_{\Gamma/M_m}(x)\| \leq\|x\|_{C^*_{\max}(\Gamma)}.
\]
The conclusion follows from Lemma~\ref{lem:willett criterion} again.
\end{proof}

\section{a-T-menability and coarse embeddability}\label{sec:aTm and CE}


In this section, we study a-T-menability of HLS groupoids.

\begin{defn}\label{defn:CDN}
A function $f:\G \to \RR$ on a groupoid $\G$ is said to be of \emph{negative type} if it satisfies the following three conditions:
\begin{itemize}
\item $f$ vanishes on the unit space $\G^{(0)}$ of $\G$;
\item $f(x)=f(x^{-1})$ for all $x\in \G$;
\item for all $a_1, a_2, \cdots, a_n \in \RR$ with $\sum_{i=1}^{n}a_i =0$ and $x_1, x_2, \cdots, x_n \in \G$ having the same range, we have $\sum_{j,k}a_j\, a_k\, f(x_j^{-1}x_k) \leq 0$.
\end{itemize}
Clearly, a conditionally negative type function has image contained in $[0,+\infty)$. 
\end{defn}

Generally, a kernel \(k:X\times X\to\RR\) on a set $X$ is said to be of \emph{negative type} if \(k(x,x)=0\), \(k(x,y)=k(y,x)\), and $\sum_{i,j} a_i a_j k(x_i,x_j)\le 0$
for any \(x_1,\cdots,x_n\in X\) and \(a_1,\ldots,a_n\in\RR\) with \(\sum_i a_i=0\).

\begin{defn}
Let \(X\) be a topological space and \(f:X\to\RR\). We say that 
\begin{enumerate}
 \item $f$ is \emph{proper} if $\{x\in X: |f(x)| \leq R\}$ is relatively compact for any $R \geq 0$;
 \item $f$ is \emph{controlled} if $f(K)$ is bounded for any compact subset $K \subseteq X$.
\end{enumerate}
\end{defn}

If $f$ is continuous, then it is automatically controlled. Moreover, if $X$ is Hausdorff, a continuous function \(f:X\to\RR\) is proper \emph{if and only if} the pre-image of any compact subset is compact. 

Now we recall the definition of a-T-menability and also introduce its weak version. For simplicity, we only consider groupoids with compact unit space.

\begin{defn}\label{defn:a-T-menable}
Let $\G$ be a locally compact Hausdorff groupoid with compact unit space. We say that $\G$ is \emph{a-T-menable} if it admits a continuous and proper function of negative type, and $\G$ is \emph{weakly a-T-menable} if it admits a controlled and proper function of negative type.
\end{defn}


The following is the full statement of Theorem~\ref{introthm:CE iff aTm}.

\begin{thm}\label{thm:CE iff aTmenable}
Let $\maN=\{N_n\}_{n\in \NN}$ be an approximating sequence of a countable discrete group $\Gamma$. Then the following are equivalent:
\begin{enumerate}
\item The box space $\boxspace_{\maN}(\Gamma)$ is coarsely embeddable into Hilbert space;
\item The associated HLS groupoid $G_{\maN}(\Gamma)$ is a-T-menable;
\item The associated HLS groupoid $G_{\maN}(\Gamma)$ is weakly a-T-menable.
\end{enumerate}
\end{thm}

The proof will be divided into several parts. In the following, fix $\maN=\{N_n\}_{n\in \NN}$ to be an approximating sequence of $\Gamma$. Denote $\Gamma_n, \pi_n$ as above for $n\in \NN \cup \{\infty\}$. For \(n\ge m\) in $\NN$, denote $q_{n,m}:\Gamma_n\to \Gamma_m$ the canonical quotient homomorphism.
Fix a proper length function \(\ell\) on \(\Gamma\). For each $n\in \NN$, equip \(\Gamma_n\) with the quotient length function $\ell_n(x)=\min\{\ell(g):\pi_n(g)=x\}$ for $x\in \Gamma_n$. 

The following local isometry property is standard (see, \emph{e.g.}, \cite[Proposition~4.4.3]{NY12} for the finitely generated case, and the same holds for countable groups equipped with proper length functions). The proof is left to the reader.

\begin{lem}\label{lem:local isom}
For every \(R>0\), there exists \(m_R\in\NN\) such that $\ell_m(\pi_m(g))=\ell(g)$ for any \(m\ge m_R\) and \(g\in\Gamma\) with \(\ell(g)\le R\). Consequently, whenever \(n\ge m\ge m_R\), we have $\ell_m(q_{n,m}(x))=\ell_n(x)$ for any \(x\in\Gamma_n\) with \(\ell_n(x)\le R\).
\end{lem}

Then we have the following:

\begin{lem}\label{lem:length function}
Define \(\hat{\ell}:G_{\maN}(\Gamma)\to[0,\infty)\) by $\hat{\ell}(n,x)=\ell_n(x)$ for $n\in \NN$ and $x\in \Gamma_n$, and $\hat{\ell}(\infty,g)=\ell(g)$ for $g\in \Gamma$. Then \(\hat{\ell}\) is continuous and proper.
\end{lem}

\begin{proof}
For continuity, it suffices to consider the fibre at infinity. For $g\in \Gamma$, Lemma~\ref{lem:local isom} shows that $\ell_n(\pi_n(g))=\ell(g)$ for sufficiently large \(n\). Hence \(\hat{\ell}\) is continuous at $(\infty,g)$ by the topology of $G_{\maN}(\Gamma)$.

To prove properness, we need to show that $\hat{\ell}^{-1}([0,R])$ is relatively compact for each $R>0$. Given $R>0$, it is easy to see
\[
\hat{\ell}^{-1}([0,R]) =\bigcup_{g\in B_\Gamma(1_\Gamma,R)}S_g \quad \text{for} \quad S_g:=\{(n,\pi_n(g)):n\in\NN\}\cup\{(\infty,g)\}.
\]
By the topology of $G_{\maN}(\Gamma)$, each $S_g$ is compact. Since  \(\ell\) is proper, \(B_\Gamma(1_\Gamma,R)\) is finite. Therefore, $\hat{\ell}^{-1}([0,R])$ is compact.
\end{proof}

The following elementary observation is a variant of \cite[Lemma~4.8]{MW26}.

\begin{lem}\label{lem:comparison of length functions}
Let \(X\) be a locally compact Hausdorff space, and let \(f,h:X\to[0,\infty)\) be two proper and controlled functions. Then there exist non-decreasing functions $\rho_-,\rho_+:[0,\infty)\to [0,\infty)$ with \(\rho_-(t)\to\infty\) as \(t\to\infty\), such that
\begin{equation}\label{EQ:control}
\rho_-(f(x))\le h(x)\le\rho_+(f(x)) \quad \text{for any} \quad x\in X .
\end{equation}
\end{lem}

\begin{proof}
Given $R\geq 0$, define
\[
\rho_+(R):=\sup_{x: f(x)\le R}h(x)\quad \text{and} \quad \rho_-(R):=\min\big\{R,\inf_{x: f(x)\ge R}h(x)\big\},
\]
where \(\sup\emptyset=0\) and \(\inf\emptyset=\infty\). 
Since $f,h$ are proper and controlled, it is clear that $\rho_+(R), \rho_-(R)$ are finite. Moreover, $\rho_+$ and $\rho_-$ are non-decreasing and satisfy \eqref{EQ:control}.

It remains to show that \(\rho_-(R)\to\infty\) as \(R\to\infty\). Fix \(C>0\). Since $h$ is proper and $f$ is controlled, there exists \(M_C>0\) such that $h(x)\le C$ implies that $f(x)\le M_C$. Hence for any $R>M_C$, $f(x)\ge R$ implies that $h(x)>C$. 
It follows that
\[
\inf_{f(x)\ge R}h(x)\ge C \quad \text{for any} \quad R>M_C. 
\]
As \(C\) is arbitrary, we obtain \(\rho_-(R)\to\infty\). 
\end{proof}

The following is the main technical ingredient in the proof of Theorem~\ref{thm:CE iff aTmenable}.

\begin{lem}\label{lem:filling lemma}
Let $\rho_-,\rho_+: [0,\infty) \to [0,\infty)$ be non-decreasing functions with \(\rho_-(t)\to\infty\) as $t\to \infty$. For each $n\in \NN$, let $\varphi_n:\Gamma_n \to \RR$ be a negative-type function such that 
\begin{equation}\label{EQ:controll for phi}
\rho_-(\ell_n(x)) \le\varphi_n(x) \le\rho_+(\ell_n(x)) \quad \text{for each} \quad  n\in \NN \quad \text{and} \quad x\in \Gamma_n.
\end{equation}
Then \(G_{\maN}(\Gamma)\) admits a continuous proper function of negative type.
\end{lem}

\begin{proof}
For each \(g\in\Gamma\) and each $n\in \NN$, note that 
\begin{equation}\label{EQ:upper bound}
\varphi_n(\pi_n(g)) \leq \rho_+(\ell_n(\pi_n(g))) \leq \rho_+(\ell(g))
\end{equation}
since $\ell_n(\pi_n(g)) \leq \ell(g)$, and functions of negative type take values in $[0,\infty)$. Hence, the sequence \(\{\varphi_n(\pi_n(g))\}_{n\in\NN}\) is bounded. Since \(\Gamma\) is countable, a diagonal argument yields a subsequence \(\{n_k\}_{k\in\NN}\) and a function \(\varphi_\infty:\Gamma \to \RR\) such that
\begin{equation}\label{EQ:limit}
\lim_{k\to\infty}\varphi_{n_k}(\pi_{n_k}(g))
=\varphi_\infty(g)
\quad\text{for any} \quad g\in\Gamma.
\end{equation}
By passing to a further subsequence if necessary, we may assume that \(n_k\ge m_k\), where \(m_k\) is given by Lemma~\ref{lem:local isom} with \(R=k\).
Since taking point-wise limits preserves being negative-type, \(\varphi_\infty\) is of negative type. Moreover, Lemma~\ref{lem:local isom} together with \eqref{EQ:controll for phi} implies that
\begin{equation}\label{EQ:estimate ncd function}
\rho_-(\ell(g))\le\varphi_\infty(g) \le\rho_+(\ell(g)) \quad\text{for any} \quad g\in\Gamma.
\end{equation}

For each \(k\in\mathbb N\), set $\varepsilon_k:=\frac{1}{\sqrt{1+\rho_-(k)}}$. For \(n_k\le n<n_{k+1}\), we define
\[
\psi_n(x):=\varphi_{n_k}(q_{n,n_k}(x))+\varepsilon_k\varphi_n(x) \quad \text{for} \quad x\in\Gamma_n.
\]
For \(n<n_1\), set \(\psi_n=\varphi_n\) and \(\psi_\infty:=\varphi_\infty\). They define a function $\psi$ on $G_{\maN}(\Gamma)$. Clearly, \(\psi\) is of negative type, since this property is preserved under pull-back by group homomorphisms, positive linear combinations and point-wise limits.

To prove $\psi$ is continuous, it suffices to consider the fibre $\Gamma_\infty$. Fixing \(g\in\Gamma\), we need to show $\psi_n(\pi_n(g)) \to \psi_\infty(g)$ as $n\to \infty$. For \(n_k\le n<n_{k+1}\), we have
\[
\psi_n(\pi_n(g)) = \varphi_{n_k}(\pi_{n_k}(g)) +\varepsilon_k\varphi_n(\pi_n(g)).
\]
By \eqref{EQ:limit}, the first term converges to \(\varphi_\infty(g)\). On the other hand, \eqref{EQ:upper bound} shows
\[
0\le\varepsilon_k\varphi_n(\pi_n(g)) \le\varepsilon_k\rho_+(\ell(g)) \to 0 \quad \text{as} \quad k\to \infty.
\]
Note that $n\to \infty$ implies $k\to \infty$. Hence, $\lim_{n\to\infty}\psi_n(\pi_n(g)) =\varphi_\infty(g)=\psi_\infty(g)$.

Finally, we prove that $\psi$ is proper. Define
\[
\tilde\rho_-(t):= \frac{\rho_-(t)}{\sqrt{1+\rho_-(t)}} \quad \text{for} \quad t\ge0.
\]
Clearly, \(\tilde\rho_-\) is non-decreasing and tends to infinity as $t\to \infty$. We claim that
\begin{equation}\label{EQ:controll for psi and rho}
\tilde\rho_-(\hat\ell(\gamma)) \le\psi(\gamma) \le2\rho_+(\hat\ell(\gamma))
\quad \text{for any} \quad \gamma\in G_{\maN}(\Gamma).
\end{equation}
To see this, let us consider \(n_k\le n<n_{k+1}\) and \(x\in\Gamma_n\).
If \(\ell_n(x)\le k\), then Lemma~\ref{lem:local isom} gives $\ell_{n_k}(q_{n,n_k}(x))=\ell_n(x)$. Hence \eqref{EQ:controll for phi} implies 
\[
\psi_n(x) \ge\varphi_{n_k}(q_{n,n_k}(x)) \ge \rho_-(\ell_{n_k}(q_{n,n_k}(x))) =  \rho_-(\ell_n(x)) \ge \tilde\rho_-(\ell_n(x)).
\]
If \(\ell_n(x)>k\), then
\[
\psi_n(x) \ge\varepsilon_k\varphi_n(x)\ge\frac{\rho_-(\ell_n(x))}{\sqrt{1+\rho_-(k)}}\ge\tilde\rho_-(\ell_n(x)).
\]
On the other hand, since \(q_{n,n_k}\) is length-decreasing and \(\varepsilon_k\le1\), we have
\[
\psi_n(x) \le(1+\varepsilon_k)\rho_+(\ell_n(x)) \le2\rho_+(\ell_n(x)).
\]
The same estimates hold for \(n<n_1\) and on $\Gamma_\infty$ by \eqref{EQ:estimate ncd function}. This proves \eqref{EQ:controll for psi and rho}.

Since \(\tilde\rho_-(t)\to\infty\), for any \(C>0\) there is \(R_C>0\) such that $\psi^{-1}([0,C]) \subseteq\hat\ell^{-1}([0,R_C])$. The right-hand side is compact by Lemma~\ref{lem:length function}. As \(\psi\) is continuous, the left-hand side is closed and hence compact. Therefore, \(\psi\) is proper.
\end{proof}

Now we are ready to prove Theorem~\ref{thm:CE iff aTmenable}. 

\begin{proof}[Proof of Theorem~\ref{thm:CE iff aTmenable}]
As before, we fix a proper length function \(\ell\) on \(\Gamma\). For each $n\in \NN$, equip \(\Gamma_n\) with the quotient length function $\ell_n$, and denote $d_n$ the length metric on $\Gamma_n$ (\emph{i.e.}, $d_n(x,y):=\ell_n(x^{-1}y)$ for $x,y\in \Gamma_n$).

``(1) $\Rightarrow$ (2)'': Suppose that $\boxspace_{\maN}(\Gamma)$ is coarsely embeddable into Hilbert space. By \cite[Theorem 11.16]{Roe03_book} (see also the proof of \cite[Proposition 11.26]{Roe03_book}), there exist non-decreasing functions \(\rho_-,\rho_+:[0,\infty)\to[0,\infty)\) with \(\rho_-(t)\to\infty\) as \(t\to\infty\), and negative-type kernels $k_n:\Gamma_n\times\Gamma_n\to [0,\infty)$ such that
\begin{equation}\label{EQ:CE double control}
\rho_-(d_n(x,y)) \le k_n(x,y) \le\rho_+(d_n(x,y)) \quad \text{for any} \quad n\in \NN \quad \text{and} \quad x,y\in \Gamma_n.
\end{equation}
For each \(n\in\mathbb N\), define
\[
\varphi_n(g)
:=
\frac{1}{|\Gamma_n|}
\sum_{x\in\Gamma_n}k_n(x,xg)
\quad \text{for} \quad  g\in\Gamma_n.
\]
We claim that \(\varphi_n\) is of negative type. Indeed, \(\varphi_n(1_{\Gamma_n})=0\) and \(\varphi_n(g^{-1})=\varphi_n(g)\). Moreover, for any \(g_1,\ldots,g_m\in\Gamma_n\) and \(a_1,\ldots,a_m\in\mathbb R\) with \(\sum_i a_i=0\), we have
\[
\sum_{i,j=1}^m a_i a_j\varphi_n(g_i^{-1}g_j)=\frac{1}{|\Gamma_n|}\sum_{x\in\Gamma_n} \sum_{i,j=1}^m a_i a_j
k_n(xg_i,xg_j)\le0.
\]
Furthermore, since \(d_n(x,xg)=\ell_n(g)\), we obtain from \eqref{EQ:CE double control} that
\[
\rho_-(\ell_n(g))\le\varphi_n(g)\le\rho_+(\ell_n(g)) \quad \text{for any} \quad n\in\NN \quad \text{and} \quad g\in\Gamma_n.
\]
Applying Lemma~\ref{lem:filling lemma}, we conclude that \(G_{\mathcal N}(\Gamma)\) admits a continuous proper function of negative type. Hence \(G_{\mathcal N}(\Gamma)\) is a-T-menable.

``(2) $\Rightarrow$ (3)'' holds trivially since every continuous function is controlled.

``(3) $\Rightarrow$ (1)'': Suppose that $G_{\mathcal N}(\Gamma)$ is weakly a-T-menable. Let $\psi:G_{\mathcal N}(\Gamma) \to [0,\infty)$ be a proper and controlled function of negative type. By Lemmas~\ref{lem:length function} and \ref{lem:comparison of length functions}, there exist non-decreasing functions \(\rho_-,\rho_+:[0,\infty)\to[0,\infty)\) with \(\rho_-(t)\to\infty\) as $t\to \infty$, such that
\[
\rho_-(\hat\ell(\gamma)) \le\psi(\gamma) \le\rho_+(\hat\ell(\gamma)) \quad \text{for any} \quad \gamma\in G_{\mathcal N}(\Gamma).
\]
For each \(n\in\NN\), denote by \(\psi_n\) the restriction of \(\psi\) to \(\Gamma_n\) and set
\[
k_n(x,y):=\psi_n(x^{-1}y) \quad \text{for} \quad x,y\in\Gamma_n.
\]
Clearly, \(k_n\) is a negative-type kernel satisfying
\[
\rho_-(d_n(x,y))\le k_n(x,y)\le\rho_+(d_n(x,y)) \quad \text{for} \quad x,y\in\Gamma_n.
\]
By \cite[Theorem 11.16]{Roe03_book}, $\boxspace_{\maN}(\Gamma)$ is coarsely embeddable into Hilbert space. 
\end{proof}

\begin{rem}
Note that the proof of ``(1) $\Rightarrow$ (2)'' above follows the outline of that for \cite[Proposition 11.26]{Roe03_book}, where Roe showed that the coarse embeddability of the box space $\boxspace_{\maN}(\Gamma)$ implies a-T-menability of the group $\Gamma$. Here we go further to a-T-menability for the whole HLS groupoid \(G_{\maN}(\Gamma)\) with the help of Lemma~\ref{lem:filling lemma}.
\end{rem}

We provide the following tool to construct a-T-menable HLS groupoids using the homology-cover construction from \cite{AGS12}, further developed in \cite{Khu14}. For a group $H$, denote its square refinement $H^{(2)}:=\langle h^2:h\in H\rangle$.

\begin{prop}\label{lem:square refinement}
Let $F$ be a finitely generated free group and let $\maN=\{N_n\}_{n\in\NN}$ be an approximating sequence. Then $\maN^{(2)}:=\{N_n^{(2)}\}_{n\in\NN}$ is an approximating sequence and $\boxspace_{\maN^{(2)}}(F)$ is coarsely embeddable into Hilbert space. Hence, $G_{\maN^{(2)}}(F)$ is a-T-menable. 
\end{prop}

\begin{proof}
Since \(N_n\) is finitely generated, \(N_n/N_n^{(2)}\) is finite. Moreover, \(N_n^{(2)}\) is characteristic in \(N_n\) and \(N_{n+1}^{(2)}\subseteq N_n^{(2)}\subseteq N_n\). Thus, \(\mathcal N^{(2)}\) is an approximating sequence.

Fix a finite free generating set of $F$ with word length $\ell$. Let $X_n$ and $\tilde X_n$ be the corresponding Cayley graphs of $F/N_n$ and $F/N_n^{(2)}$, respectively, and let $\tilde d_n$ be the graph metric on $\tilde X_n$. 
As in \cite[Section~3 and 4]{AGS12} (in particular, Propositions~3.4 and 4.3), \(\tilde X_n\) is the \(\ZZ/2\ZZ\)-homology cover of $X_n$, equipped with a left-invariant wall metric $d_{W,n}$ of negative type satisfying $d_{W,n}\le\tilde d_n$. 


By \cite[Lemma~4.4]{AGS12}, $\operatorname{girth}(X_n)\to\infty$. Moreover, \cite[Proposition~3.11]{AGS12} shows that $d_{W,n}(x,y)=\tilde d_n(x,y)$ whenever $d_{W,n}(x,y)<\operatorname{girth}(X_n)$ for $x,y\in \tilde{X}_n$. Thus, by the argument of \cite[Proposition~4.5]{AGS12}, there exists a non-decreasing function \(\rho_-:[0,\infty)\to[0,\infty)\) with \(\rho_-(t)\to\infty\) as $t\to \infty$ satisfying
\begin{equation}\label{EQ:square control}
\rho_-(\tilde d_n(x,y)) \le d_{W,n}(x,y) \le \tilde d_n(x,y) \quad \text{for any} \quad n\in \NN \quad \text{and} \quad x,y \in\tilde X_n.
\end{equation}
As in \cite[Proposition~4.3]{AGS12}, these wall metrics $d_{W,n}$ give an effective negative-type kernel on $\boxspace_{\maN^{(2)}}(F)$. Hence, $\boxspace_{\maN^{(2)}}(F)$ coarsely embeds into Hilbert space by \cite[Theorem~11.16]{Roe03_book}. Consequently, $G_{\maN^{(2)}}(F)$ is a-T-menable by Theorem~\ref{thm:CE iff aTmenable}.
\end{proof}

\begin{rem}\label{rem:easy proof}
In the setting of Proposition~\ref{lem:square refinement}, we can prove directly that $G_{\maN^{(2)}}(F)$ is a-T-menable, without invoking the filling procedure of Lemma~\ref{lem:filling lemma}. Indeed, for each $n\in \NN$, set $\varphi_n(x):=d_{W,n}(1_{F/N_n^{(2)}},x)$ for $x\in F/N_n^{(2)}$, and let \(\tilde\pi_n:F\to F/N_n^{(2)}\) be the canonical quotient map. For each $g\in F$, Lemma~\ref{lem:local isom} and the preceding local coincidence of $d_{W,n}$ and $\tilde d_n$ show that $\varphi_n(\tilde\pi_n(g))=\ell(g)$ for sufficiently large $n$.
%

Now define \(\varphi:G_{\maN^{(2)}}(F) \to [0,\infty)\) by $\varphi(n,x)=\varphi_n(x)$ for $n\in \NN$ and $\varphi(\infty,g)=\ell(g)$. Then \(\varphi\) is continuous and of negative type. Furthermore, \eqref{EQ:square control} implies that 
\[
\rho_-(\hat\ell(\gamma)) \leq\varphi(\gamma)\leq \hat\ell(\gamma)
\quad \text{for} \quad  \gamma\in G_{\mathcal N^{(2)}}(F),
\]
where \(\hat\ell\) is the function associated with \(\mathcal N^{(2)}\) in Lemma~\ref{lem:length function}. Since \(\hat\ell\) is proper by Lemma~\ref{lem:length function}, the function \(\varphi\) is proper. Thus, \(G_{\mathcal N^{(2)}}(F)\) is a-T-menable. 
\end{rem}

\begin{ex}\label{ex:AGS}
Arzhantseva, Guentner and \v{S}pakula \cite{AGS12} constructed the first example of a bounded geometry metric space that coarsely embeds into Hilbert space but does not have property A. Their example is the box space of the free group $\mathbb{F}_2$ associated with the approximating sequence $N_0=\mathbb{F}_2$ and $N_{n+1}=N_n^{(2)}$. By Theorem~\ref{thm:CE iff aTmenable}, the associated HLS groupoid is a-T-menable. Moreover, Remark~\ref{rem:easy proof} explains that the negative-type functions arising from the wall metrics in \cite{AGS12} directly yield a continuous proper function of negative type on this groupoid, without invoking the filling procedure of Lemma~\ref{lem:filling lemma}.
\end{ex}

Combining Corollary~\ref{cor:refinement preserves WCP} and Proposition~\ref{lem:square refinement}, we obtain the following:

\begin{cor}\label{cor:yes+yes}
Let $\Gamma$ be a finitely generated free group and $\maN=\{N_n\}_{n\in\NN}$ be an approximating sequence. If $G_{\mathcal N}(\Gamma)$ is metrically amenable, then $G_{\mathcal{N}^{(2)}}(\Gamma)$ is both a-T-menable and metrically amenable for $\mathcal{N}^{(2)}:=\{N_n^{(2)}\}_n$.
\end{cor}

\begin{ex}\label{ex:both properties}
Let $F$ be the free group $\mathbb{F}_2$. For each $n \in \NN$, set 
\begin{equation}\label{EQ:Wn}
 W_n=\bigcap\left\{\ker\phi ~\big|~ \phi:F\to \Gamma \text{ a group homomorphism with } |\Gamma|\le n\right\}.
\end{equation}
By \cite[Lemma 2.8]{Wil15}, $\mathcal{W}:=\{W_n\}_n$ is an approximating sequence and $G_{\mathcal W}(F)$ is metrically amenable. For $\mathcal V:=\mathcal{W}^{(2)}= \{W_n^{(2)}\}_{n\in\NN}$, Corollary~\ref{cor:yes+yes} shows that $G_{\mathcal V}(F)$ is both a-T-menable and metrically amenable.
\end{ex}

\section{a-T-menability vs metric amenability}\label{sec:Examples}

In this section, we compare a-T-menability with metric amenability for HLS groupoids. We will use property (T) and property $(\tau)$ as key tools. 

\subsection{Property (T) and property \texorpdfstring{$(\tau)$}{(tau)}}
Let $\Gamma$ be a finitely generated group and $S$ a finite symmetric generating set. Let $\pi:\Gamma \to \U(\H)$ be a unitary representation, where $\H$ is a Hilbert space and $\U(\H)$ denotes the unitary operators on $\H$. The associated \emph{normalised Laplacian} is 
\[
 \Delta_{\pi}=\Id_\H-\frac1{|S|}\sum_{s\in S}\pi(s).
\]
For a sequence $\mathcal N=\{N_n\}_{n\in\NN}$ of finite-index normal subgroups in $\Gamma$ (which is not necessarily approximating), say that $\Gamma$ has \emph{property $(\tau)$ with respect to $\maN$} if there is $\varepsilon>0$ such that $\Delta_\pi\ge\varepsilon(\Id_{\H}-P_\pi)$ for any unitary representation $\pi$ factoring through some $\Gamma/N_n$. Here $P_\pi$ is the projection onto $\pi$-invariant vectors (\emph{i.e.}, $\xi\in \H$ satisfying $\pi(g)\xi=\xi$ for any $g\in \Gamma$). Requiring the same estimate for all unitary representations gives the definition of \emph{property (T)}. See \cite{BHVBook,LA94} for more details.

The next observation follows from \cite[Propositions~4.11 and 4.15]{DW22}. We include a short proof for the convenience of the reader.

\begin{prop}\label{prop:tau metric T}
Let $\Gamma$ be a finitely generated group and $\maN$ an approximating sequence. If $G_{\maN}(\Gamma)$ is metrically amenable, then $\Gamma$ has property $(\tau)$ with respect to $\maN$ \emph{if and only if} $\Gamma$ has property (T).
\end{prop}

\begin{proof}
It suffices to prove the necessity. Let $S$ be a finite symmetric generating set of $\Gamma$ and denote $\Delta:=1-|S|^{-1}\sum_{s\in S}s\in C^*_{\max}(\Gamma)$. By Lemma~\ref{lem:willett criterion}, the representation $\Lambda=\bigoplus_n\tilde{\lambda}_{\Gamma/N_n}$ on $C^*_{\max}(\Gamma)$ is faithful. Property $(\tau)$ gives $\varepsilon>0$ such that the spectrum of $\Lambda(\Delta)$ is contained in $\{0\}\cup[\varepsilon,2]$. Since $\Lambda$ is faithful, the same spectral inclusion holds for $\Delta$ in $C^*_{\max}(\Gamma)$, and therefore for $\pi(\Delta)$ in every unitary representation $\pi$. As \(\ker\Delta_\pi\) is precisely the space of \(\pi\)-invariant vectors, we obtain $\Delta_\pi\ge\varepsilon(\Id-P_\pi)$ for every $\pi$. This shows property (T) for $\Gamma$.
\end{proof}

\begin{cor}\label{cor:tau obstruction}
Let $\Gamma$ be a finitely generated group with property $(\tau)$ relative to an approximating sequence $\maN$ and $|\Gamma| = \infty$. Then $G_{\maN}(\Gamma)$ is not a-T-menable. If additionally, $\Gamma$ does not have property (T), then $G_{\maN}(\Gamma)$ is not metrically amenable.
\end{cor}

\begin{proof}
Property $(\tau)$ makes the quotient Cayley graphs an expander family \cite[Theorem~4.3.2]{LA94}, and note that their sizes tend to infinity because $\Gamma$ is infinite. Hence, their coarse disjoint union cannot coarsely embed into Hilbert space. Theorem~\ref{thm:CE iff aTmenable} implies the first assertion and Proposition~\ref{prop:tau metric T} implies the second.
\end{proof}


Corollary~\ref{cor:tau obstruction} can be applied to produce HLS groupoids which are neither a-T-menable nor metrically amenable. We provide two such examples. 

\begin{ex}\label{ex:F2 non+non}
Let $F$ be the subgroup of $\SL_2(\ZZ)$ generated by
\[
A=\begin{pmatrix}1&2\\0&1\end{pmatrix},
\qquad
B=\begin{pmatrix}1&0\\2&1\end{pmatrix}.
\]
By Sanov's theorem, \(F\) is a free group of rank two and has index \(12\) in \(\SL_2(\ZZ)\) (see, \emph{e.g.}, \cite[Theorems~1 and~2]{CGS17}). For $m\ge1$, let
\[
\Gamma(m):=\ker\bigl( \SL_2(\ZZ)\longrightarrow\SL_2(\ZZ/m\ZZ) \bigr).
\]
By Selberg's theorem (see \cite[Example~4.3.3(D) and Section~4.4]{LA94}) and the finite-index inheritance of property $(\tau)$ (see \cite[Theorem~8]{AE12}), 
$F$ has property $(\tau)$ with respect to the family $\{F\cap\Gamma(m)\}_{m\in \NN}$. Now set
\[
C_n:=F\cap\Gamma(2^n)
\quad \text{and} \quad
\mathcal C=\{C_n\}_{n\in \NN}.
\]
Then $\mathcal C$ is an approximating sequence of \(F\) consisting of congruence subgroups, and $F$ has property $(\tau)$ with respect to $\mathcal C$. Since $F\cong\mathbb F_2$ does not have property (T), Corollary~\ref{cor:tau obstruction} shows that $G_{\mathcal C}(F)$ is neither a-T-menable nor metrically amenable.
\end{ex}


We next give an example of a group having property $(\tau)$ with respect to every approximating sequence.

\begin{ex}\label{ex:SL non+non}
Let $\Gamma=\SL_2(\ZZ[1/p])$, where $p$ is a prime number. By \cite[Example~4.3.3(E)]{LA94}, $\Gamma$ has property $(\tau)$ with respect to any approximating sequence. On the other hand, $\Gamma$ has the Haagerup property since it is a
countable subgroup of $\mathrm{GL}_2(\mathbb Q)$ (see \cite[Theorem~4]{GHW05}), and hence does not have property (T). Therefore, Corollary~\ref{cor:tau obstruction} implies that for any approximating sequence $\mathcal N$ of $\Gamma$, the associated HLS groupoid $G_{\mathcal N}(\Gamma)$ is neither a-T-menable nor metrically amenable.
\end{ex}

\subsection{Metric amenability does not imply a-T-menability}

Examples~\ref{ex:F2 non+non} and \ref{ex:SL non+non} show that property $(\tau)$ obstructs a-T-menability. However, for groups without property (T), it also rules out metric amenability by Proposition~\ref{prop:tau metric T}. To separate these two properties, we now construct an example using the idea of relative expansion from \cite{AT15}.

Throughout this subsection, we will use the following notation. Let $F:=\FF(a,b)$ be the free group generated by two elements $a,b$, and $S:=\{a,a^{-1},b,b^{-1}\}$ be the generating set. Denote
\[
 q:F \longrightarrow \SL_3(\ZZ) \quad \text{by} \quad 
 q(a)=\I_3+E_{12} \quad \text{and} \quad 
 q(b)=\begin{pmatrix}0&0&1\\1&0&0\\0&1&0\end{pmatrix}.
\]
Here $\I_3$ is the $3$-by-$3$ identity matrix and $E_{ij}$ is the matrix with $1$ in the $(i,j)$-position and $0$ otherwise, for $1 \leq i,j \leq 3$. It is routine to check that $q$ is surjective. Given a prime number $p$, denote $\FF_p:=\ZZ/p\ZZ$, $Q_p:=\SL_3(\FF_p)$. Let \(q_p:F\to Q_p\) be the composition of \(q\) with reduction modulo \(p\).


For each $i \in \NN$, set $W_i$ as in \eqref{EQ:Wn}. Starting with $M_0:=F$, we define recursively a sequence of subgroups $\{M_i\}_i$ of $F$ as follows. Setting $U_i:=W_i\cap M_{i-1}$, choose a prime number $p_i > 2|F/U_i|$ and set $M_i:=U_i\cap\ker q_{p_i}$.

\begin{thm}\label{thm:counterexample}
The above chain $\M=\{M_i\}_i$ is an approximating sequence of $F$ such that the associated HLS groupoid $G_{\M}(F)$ is metrically amenable but not a-T-menable.
\end{thm}

To prove Theorem~\ref{thm:counterexample}, we need some preparations. Let us start with the following elementary fact in group theory.

\begin{lem}\label{lem:finite group prodduct}
Let $U \trianglelefteq F$ be a normal subgroup with finite index and put $R:=F/U$. If $p>|R|$, then the map
\[
F/(U\cap\ker q_p) \longrightarrow Q_p\times R, \quad g(U\cap\ker q_p)\mapsto (q_p(g), gU) \quad \text{for} \quad g\in F
\]
is an isomorphism. 
\end{lem}

\begin{proof}
Consider the map $\Phi: F \to Q_p\times R$ defined by $g\mapsto (q_p(g), gU)$. Clearly, $\ker \Phi = U\cap\ker q_p$. Hence it suffices to show that $\Phi$ is surjective. 

Claim: $q_p(U)=Q_p$. In fact, note that $q_p(U)$ is a normal subgroup of $Q_p$ and the quotient map $F/U \to Q_p/q_p(U)$ is surjective. Hence, we have $|Q_p/ q_p(U)| \leq |R| < p$. Now consider the quotient map $\theta: Q_p \to Q_p/ q_p(U)$. For each elementary matrix $\I_3 + E_{ij}$ with $i\neq j$, we have $(\I_3 + E_{ij})^k = \I_3 + kE_{ij}$ in $Q_p$ for each $k \in \NN$. Hence, $\I_3 + E_{ij}$ has order $p$ in $Q_p$ and therefore, the order of $\theta(\I_3 + E_{ij})$ is either $1$ or $p$. However, $|Q_p/ q_p(U)| < p$ and hence, $\theta(\I_3 + E_{ij}) = 1$ in $Q_p/ q_p(U)$.  Since $Q_p$ is generated by these elementary matrices, we conclude that $\theta$ is trivial. Therefore, $q_p(U)=Q_p$.

The claim implies that $U/(U\cap\ker q_p) \cong q_p(U) = Q_p$. Hence, we obtain
\[
[F:U\cap\ker q_p] =[F:U]\cdot [U:U\cap\ker q_p]=|R|\cdot|Q_p|.
\]
Since these are finite groups, we conclude that $\Phi$ is surjective.
\end{proof}

We also need the following fact in representation theory.

\begin{lem}\label{lem:rep of Qp}
Every non-trivial complex irreducible representation of $Q_p$ has dimension at least $p-1$.
\end{lem}

\begin{proof}
Fix a non-trivial complex irreducible representation $\rho: Q_p \to \GL(V)$. 
Let $u:=\I_3 + E_{12} \in Q_p$, and it is routine to check that the normal subgroup generated by $u$ coincides with $Q_p$. Hence, $A:=\rho(u) \neq \Id_V$. Since $u^p=\I_3$, we have $A^p = \Id_V$. Then $A$ is diagonalizable and has an eigenvalue $\zeta^k$ for $\zeta:=e^{2\pi i/p}$ and $k \in \{1,\cdots, p-1\}$. 

For $t\in \FF_p^\times$, denote 
\[
D_t=\begin{pmatrix}t&0&0\\0&1&0\\0&0&t^{-1}\end{pmatrix} \in Q_p.
\]
Direct calculations show that  $D_tuD_t^{-1}=\I_3 + tE_{12} = u^t$, which implies that $A$ and $A^t$ have the same eigenvalues. Therefore, $\zeta, \zeta^2, \cdots, \zeta^{p-1}$ are eigenvalues of $A$, which implies that $\dim V \geq p-1$. 
\end{proof}

Now we use property (T) to bound the spectral gap of certain representations of $F$. Let us fix some notation. Recall that for a unitary representation $\pi:F \to \U(\H)$, the associated normalised Laplacian is 
\[
 \Delta_{\pi}=\Id_\H-\frac1{|S|}\sum_{s\in S}\pi(s).
\]
Also define the \emph{energy} to be
\[
\E_{\pi}(v)=\langle\Delta_{\pi} v,v\rangle
 =\frac1{2|S|}\sum_{s\in S}\|\pi(s)v-v\|^2 \quad \text{for} \quad v\in \H.
\]
We will use the same notation for unitary representations of quotients $F/N$ with the image of $S$ as the finite generating set of $F/N$. Repeated images are counted with multiplicity.

Since $\SL_3(\ZZ)$ has property (T) (see, \emph{e.g.}, \cite[Theorem 4.2.5]{BHVBook}), there exists a constant $\lambda>0$ such that every representation of $F$ factoring through $q$ satisfies
\begin{equation}\label{eq:kazhdan}
 \Delta_\pi\ge\lambda(\Id-P_\pi).
\end{equation}
Fix such a $\lambda$ in the sequel.

The following result is crucial, providing a uniform spectral gap for certain representations. 

\begin{lem}\label{lem:tensor spectral gap}
Let $\rho: F \to \U(\H_\rho)$ be a $d$-dimensional irreducible unitary representation factoring through $q$, and $\sigma: F \to \U(\H_\sigma)$ be any $m$-dimensional unitary representation. Then we have
\[
 \Delta_{\rho\otimes\overline\sigma}
 \ge\frac\lambda2\left(1-\frac md\right)\Id_{\H_\rho \otimes \overline{\H_\sigma}}.
\]
\end{lem}

Let us fix some notation. For finite-dimensional Hilbert spaces $\H_1$ and $\H_2$, denote $\HS(\H_2,\H_1)$ the set of all Hilbert-Schmidt operators with the inner product $\langle T,S \rangle_{\HS}:=\tr(S^*T)$ and norm $\|T\|_{\HS}$. If $\H_1 = \H_2 = \H$, denote $\HS(\H) = \HS(\H,\H)$. Identify $\H_1 \otimes \overline{\H_2}$ with $\HS(\H_2, \H_1)$ by sending each elementary tensor $v \otimes \overline{w}$ to the operator $T_{v,w}: \xi \mapsto \langle \xi,w \rangle v$. Equipping $\H_1 \otimes \overline{\H_2}$ with the tensor norm, this provides an isometric isomorphism from $\H_1 \otimes \overline{\H_2}$ to $\HS(\H_2,\H_1)$. 

%

\begin{proof}[Proof of Lemma~\ref{lem:tensor spectral gap}]
Under the identification of $\H_\rho \otimes \overline{\H_\sigma}$ and $\HS(\H_\sigma, \H_\rho)$ as above, $\rho\otimes\overline\sigma$ corresponds to the unitary representation $\pi: F \to \U(\HS(\H_\sigma, \H_\rho))$ defined by $\pi(g)(T) = \rho(g) T \sigma(g)^\ast$ for $g \in F$ and $T \in \HS(\H_\sigma, \H_\rho)$. 
Consider the adjoint representation $\Ad \rho: F \to \U(\HS(\H_\rho))$ defined by
\[
\big((\Ad \rho)(g)\big)(A):=\rho(g) A \rho(g)^* \quad \text{for} \quad g\in F \quad \text{and}  \quad A \in \HS(\H_\rho).
\]
Since $\rho$ factors through $q$, so does $\Ad \rho$. Hence, \eqref{eq:kazhdan} shows that 
\begin{equation}\label{EQ:Lap}
\Delta_{\Ad \rho} \geq \lambda(\Id_{\HS(\H_\rho)}-P_{\Ad \rho}). 
\end{equation}
Since $\rho$ is irreducible, Schur's lemma shows that $\Ad \rho$-invariant vectors are $\CC \cdot \Id_{\H_\rho}$.

Fix $T \in \HS(\H_\sigma, \H_\rho)$ with $\|T\|_{\HS} =1$, and denote $B:=(TT^*)^{1/2} \geq 0$. Then
\[
\|B\|_{\HS} = \tr(B^2)^{1/2} = \tr(TT^*)^{1/2} = \|T\|_{\HS} = 1.
\]
Also $\rank(B) = \rank(TT^*) = \rank(T) \leq m$. Hence, we have
\begin{equation}\label{EQ:trB}
\tr(B)^2 \leq \rank(B) \cdot \tr(B^2) = \rank(B) \leq m.
\end{equation}
On the other hand, note that
\[
P_{\Ad \rho}(B) = \frac{\langle B, \Id_{\H_\rho} \rangle_{\HS}}{\|\Id_{\H_\rho}\|_{\HS}^2} \cdot \Id_{\H_\rho} = \frac{\tr(B)}{d} \cdot  \Id_{\H_\rho}.
\]
Combining with \eqref{EQ:trB}, we obtain
\[
\big\|(\Id_{\HS(\H_\rho)}-P_{\Ad \rho})B\big\|_{\HS}^2 =\|B\|_{\HS}^2-\big\|P_{\Ad \rho}(B)\big\|_{\HS}^2 =1-\frac{(\tr(B))^2}{d} \geq 1-\frac{m}{d}.
\]
Combining the above inequality with \eqref{EQ:Lap}, we obtain
\begin{align}\label{EQ:engery estimate}
\E_{\Ad \rho}(B) &=\langle\Delta_{\Ad \rho}B,B\rangle_{\HS} \geq \lambda\langle(\Id_{\HS(\H_\rho)}-P_{\Ad \rho})B, B\rangle_{\HS}\\
&=\lambda\|(\Id_{\HS(\H_\rho)}-P_{\Ad \rho})B\|_{\HS}^2 \geq \lambda\left(1-\frac md\right). \notag
\end{align}

Finally, for each $s\in S$ note that
\[
\big((\rho(s)T\sigma(s)^*) \cdot (\rho(s)T\sigma(s)^*)^*\big)^{1/2} = (\rho(s)B^2\rho(s)^*)^{1/2} = \rho(s)B\rho(s)^*.
\]
Thus, \cite[Theorem 1]{AY81} shows that $\|\rho(s)B\rho(s)^*-B\|_{\HS}^2 \leq 2 \|\rho(s)T\sigma(s)^*-T\|_{\HS}^2$. Hence,
\[
\E_{\Ad \rho}(B) =\frac1{2|S|} \sum_{s\in S} \big\|\rho(s)B\rho(s)^*-B\big\|_{\HS}^2 \leq \frac1{|S|} \sum_{s\in S} \big\|\rho(s)T\sigma(s)^*-T\big\|_{\HS}^2 = 2\E_{\pi}(T).
\]
Combining with \eqref{EQ:engery estimate}, we conclude $\E_{\pi}(T) \geq \frac\lambda2\left(1-\frac md\right)$ for any $T \in \HS(\H_\sigma, \H_\rho)$ with $\|T\|_{\HS} =1$. This finishes the proof. 
\end{proof}

For the sequences $\{M_i\}_i$, $\{U_i\}_i$ and $\{p_i\}_i$ chosen at the beginning of this section, Lemma~\ref{lem:finite group prodduct} shows that $F/M_i \cong Q_{p_i} \times (F/U_i)$ for each $i$. Denote $K_i:=Q_{p_i} \times \{1_{F/U_i}\} \cong U_i/M_i$ as a subgroup of $F/M_i$. We have the following relative expansion result.

\begin{cor}\label{cor:spetral gap}
With the same notation as above, for each $i \in \NN$ and unitary representation $\pi: F/M_i \to \U(\H)$, we have
\begin{equation}\label{EQ:spectral gap final}
\Delta_\pi\geq \frac{\lambda}{4}(\Id_{\H}-P_{K_i}), \quad \text{where} \quad P_{K_i}=|K_i|^{-1}\sum_{k\in K_i}\pi(k).
\end{equation}
\end{cor}

\begin{proof}
Decomposing $\pi$ into a direct sum of irreducible representations, we can assume that $\pi$ is irreducible. Then $\pi$ has the form $\rho\otimes\overline{\sigma}$, where $\rho: Q_{p_i} \to \U(\H_\rho)$ and $\sigma: F/U_i \to \U(\H_\sigma)$ are irreducible unitary representations. 

Recall that $K_i\cong Q_{p_i}\times\{1_{F/U_i}\}$, so $\pi|_{K_i}=\rho\otimes\Id_{\overline{\H_\sigma}}$.
If $\rho$ is trivial, then $P_{K_i}=\Id_{\H}$ and hence, \eqref{EQ:spectral gap final} holds. Now suppose that $\rho$ is non-trivial. Since $\rho$ is irreducible, it has no non-zero invariant vectors. Consequently, $P_{K_i}=0$. Denote $\rho':=\rho \circ q_{p_i}$, and $\sigma': F \to \U(\H_\sigma)$ by $g \mapsto \sigma(gU_i)$. Clearly, $\rho'$ factors through $q$ and remains irreducible since $q_{p_i}$ is surjective. Lemma~\ref{lem:rep of Qp} and the choice of $p_i$ imply
\[
\dim\H_\rho\geq p_i-1\geq 2|F/U_i| \geq 2\dim\H_\sigma.
\]
Applying Lemma~\ref{lem:tensor spectral gap} to $\rho'\otimes\overline{\sigma'}$, we obtain $\Delta_{\rho'\otimes
\overline{\sigma'}} \geq \frac{\lambda}{4}\Id_{\H}$. Since $(\rho'\otimes\overline{\sigma'})(g) = \pi(gM_i)$ for each $g\in F$, 
we obtain
\[
\Delta_\pi = \Delta_{\rho'\otimes \overline{\sigma'}} \geq \frac{\lambda}{4}\Id_{\H} =\frac{\lambda}{4}(\mathrm{Id}_{\mathcal H}-P_{K_i}).
\]
This finishes the proof.
\end{proof}

The estimate in Corollary~\ref{cor:spetral gap} gives a uniform spectral gap relative to the subgroups \(K_i\), and hence a relative Poincar\'{e} inequality in the sense of Arzhantseva and Tessera \cite{AT15}. Also it is not hard to see that non-trivial elements of \(K_i\) escape to infinity. Thus, it follows from \cite[Proposition~2.3]{DK18} and \cite[Theorem 4]{Tes09} that the box space $\boxspace_{\M}(F)$ is not coarsely embeddable into Hilbert space. Applying Theorem~\ref{thm:CE iff aTmenable}, we obtain that $G_{\M}(F)$ is not a-T-menable. For the convenience of the reader, we give a direct proof below.


\begin{proof}[Proof of Theorem~\ref{thm:counterexample}]
First, it is clear from the construction that each $M_i$ is normal in $F$ and has finite index. Also $M_i\subseteq U_i\subseteq M_{i-1}$ and $M_i\subseteq W_i$, so the chain is nested and $\bigcap_i M_i = \{1_F\}$. Hence, $\mathcal M$ is an approximating sequence of $F$. 
Moreover, \cite[Lemma 2.8]{Wil15} shows that $G_{\mathcal{W}}(F)$ is metrically amenable for $\mathcal W=\{W_i\}_i$. Applying Corollary~\ref{cor:refinement preserves WCP}, we obtain that $G_{\M}(F)$ is also metrically amenable.


Next, we show that $G_{\M}(F)$ does not have the Haagerup property. Suppose that $G_{\M}(F)$ admits a continuous proper function $\psi$ of negative type, and let $\psi_i$ be its restriction to $F/M_i$ for $i\in \NN$. For each $g\in F$, the topology of $G_{\M}(F)$ implies that $\psi_i(gM_i) \to \psi(\infty,g)$ as $i\to \infty$. Hence, 
\begin{equation}\label{eq:bounded}
 C:=\sup_{i\in \NN}\frac1{|S|}\sum_{s\in S}\psi_i(sM_i)<\infty.
\end{equation}

For each $i\in \NN$, since $\psi_i$ is a negative-type function on $G_i:=F/M_i$ with $\psi_i(1_{G_i}) = 0$, the classical GNS construction provides a unitary representation $\pi_i: G_i \to \U(\H_i)$ and a 1-cocycle $b_i: G_i \to \H_i$ such that $\psi_i(g) = \|b_i(g)\|^2$. For $v_i:=\frac1{|G_i|}\sum_{h\in G_i}b_i(h)\in\H_i$, the cocycle condition of $b_i$ implies that $\pi_i(g)v_i + b_i(g) = v_i$ for each $g \in G_i$, which implies that $\psi_i(g) = \|\pi_i(g)v_i - v_i\|^2$. Hence for $K_i$ defined above, we have
\begin{align}\label{EQ:estimate 2}
\frac1{|K_i|}\sum_{k\in K_i}\psi_i(k) &= \frac1{|K_i|}\sum_{k\in K_i}\|\pi_i(k)v_i - v_i\|^2 = \frac1{|K_i|}\sum_{k\in K_i} \big(2\|v_i\|^2 -2\mathrm{Re}\langle\pi_i(k)v_i,v_i\rangle\big) \\
&= 2\|v_i\|^2 - 2\langle P_{K_i}v_i,v_i\rangle = 2\|v_i\|^2 - 2\|P_{K_i}v_i\|^2=2\|(\I_{\H_i}-P_{K_i})v_i\|^2.\notag
\end{align}
Now Corollary~\ref{cor:spetral gap} shows that
\begin{equation}\label{EQ:estimate 3}
2\|(\I_{\H_i}-P_{K_i})v_i\|^2 \leq \frac8\lambda\langle\Delta_{\pi_i}v_i,v_i\rangle = \frac{4}{\lambda|S|}\sum_{s\in S}
\|\pi_i(sM_i)v_i-v_i\|^2 = \frac{4}{\lambda|S|}\sum_{s\in S} \psi_i(sM_i).
\end{equation}
Combining \eqref{eq:bounded}, \eqref{EQ:estimate 2} and \eqref{EQ:estimate 3}, we obtain
\[
\frac1{|K_i|}\sum_{k\in K_i}\psi_i(k) \leq \frac{4}{\lambda|S|}\sum_{s\in S} \psi_i(sM_i) \leq \frac{4C}{\lambda}.
\]
Since $|K_i|\ge2$ and $\psi_i(e)=0$, we can choose $k_i\in K_i\setminus\{e\}$ such that
\begin{equation}\label{eq:point}
 \psi_i(k_i)\le \frac{8C}{\lambda}.
\end{equation}
On the other hand, note that $K_i = U_i/M_i$ and $|k_i|_{G_i}= \min\{|g|_S:gM_i=k_i\}$. For any $g\in F$ with $gM_i=k_i$, we have $g\in U_i \setminus M_i \subseteq W_i \setminus \{1_F\}$. Since $\{W_i\}_i$ is an approximating sequence, we have
\[
r_i:=\min\{|g|_S:g\in W_i\setminus\{1_F\}\}\to \infty \quad \text{as} \quad i \to \infty.
\]
By Lemma~\ref{lem:length function}, the continuous length function $\hat\ell$ on $G_{\M}(F)$ satisfies
\[
\hat\ell(i,k_i)=|k_i|_{G_i}\ge r_i\to \infty.
\]
Hence the set $\{(i,k_i):i\in\NN\}$ is not relatively compact. This contradicts the properness of $\psi$, since \eqref{eq:point} gives $\psi(i,k_i)=\psi_i(k_i)\le8C/\lambda$ for every $i$.
\end{proof}

\begin{rem}\label{rem:no tau M}
By Proposition~\ref{prop:tau metric T}, $F$ does not have property $(\tau)$ with respect to $\M$. This is consistent with Corollary~\ref{cor:spetral gap}, which provides a uniform spectral gap only on the orthogonal complement of the \(K_i\)-invariant vectors.
\end{rem}

\begin{rem}\label{rem:weakly embedded}
Note that the associated box space $\boxspace_{\M}(F)$ does not weakly contain expanders (see \cite[Definition 1.3]{DK18} for the definition). Hence, it provides another example that neither weakly contains expanders nor coarsely embeds into Hilbert space. The first such example was discovered in \cite{AT15} (see also \cite{DK18}).

To see this, take $N_0=F$, $N_{n+1} = N_n^{(2)}$ as in Example~\ref{ex:AGS}, and $d_n=[F:N_n]$. Given $i \in \NN$, set $n(i)=\max\{n\ge0:d_n\le i\}$ and $B_i=N_{n(i)}$.
Then $n(i)\to\infty$ and $\mathcal B=\{B_i\}_i$ is an approximating sequence whose box space coarsely embeds into Hilbert space by \cite{AGS12}. Since $[F:B_i]\leq i$, the definition of $W_i$ gives $M_i\subseteq W_i\subseteq B_i$.
As $M_i$ and $B_i$ are finite-index normal subgroups of $F$, the assertion follows from \cite[Proposition~2.4]{DK18}.
\end{rem}

\subsection{A-T-menability does not imply metric amenability}\label{subsec:reverse}

We follow the notation from Example~\ref{ex:F2 non+non}. Let $F=\langle A,B\rangle\cong\mathbb F_2$ and $\mathcal C=\{C_n\}_{n\in\NN}$ be the congruence sequence from Example~\ref{ex:F2 non+non}. Set $S:=\{A,A^{-1},B,B^{-1}\}$. 
We show the following:

\begin{thm}\label{thm:reverse counterexample}
For $C_n$ defined in Example~\ref{ex:F2 non+non}, set $D_n:=C_n^{(2)}$ and $\mathcal D:=\{D_n\}_{n\in\NN}$. Then $G_{\mathcal D}(F)$ is a-T-menable but not metrically amenable.
\end{thm}

To prove it, we first establish a spectral-gap estimate for elementary abelian $2$-group extensions. Let us fix some notation.

Let $\Gamma$ be a finitely generated group with a finite symmetric generating set $S$. For a finite quotient $H$ of $\Gamma$, denote $q_H: \Gamma \to H$ the canonical quotient map and $\tilde{\lambda}_H=\lambda_H \circ q_H: \Gamma \to \B(\ell^2(H))$, where $\lambda_H$ denotes the left regular representation of $H$. 
Given a character $\chi:\Gamma\to\mathbb T$, the normalised Laplacian of the twisted representation
$\overline\chi\otimes\tilde{\lambda}_H$ is denoted by
\[
\Delta_{\chi,H}:=\Id-\frac1{|S|}\sum_{s\in S}\overline{\chi}(s)\tilde{\lambda}_H(s).
\]
Clearly, for any $f\in \ell^2(H)$ we have
\begin{equation}\label{eq:twisted energy}
\langle\Delta_{\chi,H}f,f\rangle
=\frac1{2|S|}\sum_{s\in S}\|\tilde{\lambda}_H(s)f-\chi(s)f\|_2^2.
\end{equation}

The following is the main technical ingredient in the proof of Theorem~\ref{thm:reverse counterexample}.

\begin{lem}\label{lem:twisted lifting}
Let $H,Q$ be finite quotients of $\Gamma$ and $p:H\to Q$ be a surjective homomorphism satisfying $q_Q=p\circ q_H$ such that $\ker p$ is an elementary abelian $2$-group (i.e., $(\ZZ/2\ZZ)^r$ for some $r\ge0$).
If $\chi:\Gamma\to\mathbb T$ is a character and $\Delta_{\chi^2,Q}\ge\delta \cdot \Id$ for some $\delta>0$, then $\Delta_{\chi,H}\ge\frac{\delta}{4} \cdot \Id$.
\end{lem}

\begin{proof}
Let \(E=\ker p\) and \(\mu\) be the smallest eigenvalue of \(\Delta_{\chi,H}\). For each \(k\in E\), define the right translation 
\[
(R_k f)(x):=f(xk)\quad \text{for} \quad f\in \ell^2(H) \quad \text{and} \quad x\in H.
\]
Since \(E\) is an elementary abelian \(2\)-group, \(\{R_k:k\in E\}\) are commuting self-adjoint unitaries and also commute with \(\Delta_{\chi,H}\). Hence, \(\Delta_{\chi,H}\) admits a $\mu$-eigenvector \(f \in \ell^2(H)\) with $\|f\|_2=1$ such that for each $k \in E$, there exists $\varepsilon_k\in\{1,-1\}$ satisfying
\[
f(xk)=\varepsilon_k f(x) \quad \text{for any} \quad x\in H.
\]

Define \(T:\CC\to\CC\) by \(T(0)=0\) and $T(z):=z^2/|z|$ for $z \neq 0$. Then we have
\begin{equation}\label{EQ:property of T}
|T(z)|=|z|,\quad T(-z)=T(z), \quad \text{and} \quad T(\omega z)=\omega^2T(z) \quad \text{for any} \quad \omega\in\mathbb T.
\end{equation}
Moreover, \(T\) is \(2\)-Lipschitz. Now define \(u:Q\to\CC\) by
\[
u(p(x)):=\sqrt{|E|}\cdot T(f(x)) \quad \text{for} \quad x\in H.
\]
This is well-defined since $T(f(xk))=T(\varepsilon_k f(x))=T(f(x))$ for any $k \in E$. Moreover, since \(|T(z)|=|z|\) and \(|f|\) is constant on each right \(E\)-coset, we have
\[
\|u\|_2^2 =|E|\cdot \sum_{y\in Q}|T(f(x_y))|^2=\sum_{x\in H}|f(x)|^2=1,
\]
where \(x_y \in H\) is any representative of \(y\in Q\).

For \(s\in S\) and \(x\in H\), the properties of $T$ in \eqref{EQ:property of T} and $T$ being $2$-Lipschitz give
\begin{align*}
|(\tilde{\lambda}&_Q(s)u)(p(x)) -\chi(s)^2u(p(x))| 
=\sqrt{|E|} \cdot \left|T(f(q_H(s)^{-1}x))-T(\chi(s)f(x))\right|\\
&\le2\sqrt{|E|}\cdot \left|f(q_H(s)^{-1}x)-\chi(s)f(x)\right|=2\sqrt{|E|}\cdot |(\tilde{\lambda}_H(s)f)(x)-\chi(s)f(x)|.
\end{align*}
Squaring and summing over $x\in H$, we obtain
\[
|E| \cdot \|\tilde{\lambda}_Q(s)u-\chi(s)^2u\|_2^2 =
\sum_{x\in H}|(\tilde{\lambda}_Q(s)u)(p(x))-\chi(s)^2u(p(x))|^2\le
4|E|\cdot \|\tilde{\lambda}_H(s)f-\chi(s)f\|_2^2.
\]
Using \eqref{eq:twisted energy} and the assumption that $\Delta_{\chi^2,Q}\ge\delta\Id$, we obtain
\[
\delta = \delta \cdot \|u\|_2^2 \leq \langle\Delta_{\chi^2,Q}u,u\rangle \le4\langle\Delta_{\chi,H}f,f\rangle =4\mu.
\]
Hence, we conclude $\Delta_{\chi,H}\ge\frac{\delta}{4}\cdot \mathrm{Id}$.
\end{proof}

Now we prove Theorem~\ref{thm:reverse counterexample}.

\begin{proof}[Proof of Theorem~\ref{thm:reverse counterexample}]
By Proposition~\ref{lem:square refinement}, $\mathcal D$ is an approximating sequence and $G_{\mathcal D}(F)$ is a-T-menable. It remains to show that it is not metrically amenable. Put
\[
H_n:=F/D_n,\quad Q_n:=F/C_n\quad \text{and} \quad  E_n:=C_n/D_n.
\]
Since $C_n$ is finitely generated and every element of $E_n$ has order at most $2$, $E_n$ is an elementary abelian $2$-group. Also, $Q_n$ is a finite $2$-group. Indeed, since $A$ and $B$ are congruent to the identity modulo $2$, $Q_n$ is contained in the kernel of $\SL_2(\ZZ/2^n\ZZ) \to \SL_2(\ZZ/2\ZZ)$. It is a $2$-group as any element has order dividing $2^{n-1}$.

Reduction modulo $3$ maps $F$ onto $\mathrm{SL}_2(\mathbb F_3)$, since the reductions of \(A^{-1}\) and \(B^{-1}\) are the standard elementary matrices \(E_{12}(1)\) and \(E_{21}(1)\), which generate \(\SL_2(\mathbb F_3)\). Consider
\[
\SL_2(\mathbb F_3) \longrightarrow \mathrm{PSL}_2(\mathbb F_3)\cong A_4 \longrightarrow A_4/V_4\cong\ZZ/3\ZZ
\longrightarrow\mathbb T,
\]
where $A_4$ is the alternating group and $V_4$ is the Klein four group. Here the first isomorphism follows from \cite[Section 11.2.1]{Bonnafe11}, and the last arrow is a non-trivial character sending $\bar{1}$ to $e^{2\pi i/3}$. Composing with reduction modulo $3$, we obtain a character $\chi:F\to\mathbb T$ of order $3$.

For each $n \in \NN$, consider the unitary representation 
\[
\pi_n:=\overline{\chi}^2 \otimes \tilde{\lambda}_{Q_n}: F \longrightarrow \U(\CC \otimes \ell^2(Q_n)) \cong \U(\ell^2(Q_n)). 
\]
Since \(\chi\) factors through reduction modulo \(3\) and \(\tilde{\lambda}_{Q_n}\) factors through reduction modulo \(2^n\), the representation \(\pi_n\) factors through \(F/(F\cap\Gamma(3\cdot2^n))\). 

Moreover, $\pi_n$ has no non-zero invariant vectors. Indeed, if $v\ne0$ were invariant, then $\tilde{\lambda}_{Q_n}(g)v =\chi(g)^2v$ for any $g\in F$.
Taking \(g\in C_n\), we obtain \(v=\chi(g)^2v\) and hence \(\chi(g)^2=1\). Thus \(C_n\subseteq\ker\chi^2\), which implies that \(\chi^2\) factors through \(Q_n\). This is impossible since \(Q_n\) is a finite \(2\)-group whereas \(\chi^2\) has order \(3\).

As explained in Example~\ref{ex:F2 non+non}, $F$ has property $(\tau)$ with respect to the family $\{F\cap\Gamma(m)\}_{m\in \NN}$. Since \(\pi_n\) factors through \(F/(F\cap\Gamma(3\cdot2^n))\) as shown above, there exists $\delta>0$ such that $\Delta_{\chi^2,Q_n} =\Delta_{\pi_n}\ge\delta \cdot \Id$ for any $n\in\NN$. Applying Lemma~\ref{lem:twisted lifting}, we obtain $\Delta_{\chi,H_n}\ge\frac{\delta}{4} \cdot \Id$ for any $n\in \NN$.

Finally, consider the self-adjoint element
\[
a:=1-\frac1{|S|}\sum_{s\in S}\overline{\chi(s)}s \in \CC[F].
\]
Then direct calculations show
\[
\tilde{\lambda}_{H_n}(a) =\Delta_{\chi,H_n} \ge\frac{\delta}{4} \cdot \Id \quad \text{for every} \quad n\in \NN.
\]
If $G_{\mathcal D}(F)$ were metrically amenable, Lemma~\ref{lem:willett criterion} would imply that the representation  $\Lambda:=\bigoplus_n\tilde{\lambda}_{H_n}$ is faithful on $C^*_{\max}(F)$. Hence $a \ge \frac{\delta}{4} \cdot 1_{F}$ in $C^*_{\max}(F)$.
However, the one-dimensional representation $\chi$ satisfies
\[
\chi(a) =1-\frac1{|S|}\sum_{s\in S} \overline{\chi(s)}\chi(s)=0,
\]
a contradiction. Therefore, $G_{\mathcal D}(F)$ is not metrically amenable.
\end{proof}

\section{Conclusion}\label{sec:conclusions}

We have constructed different HLS groupoids over the same free group
$F=\mathbb F_2$, which exhibit all four possible combinations
of a-T-menability and metric amenability.
The results are summarised in Table~\ref{tab:four-cases}.
None of these groupoids is topologically amenable,
since $\mathbb F_2$ is non-amenable
\cite[Lemma~2.4]{Wil15}.


\begin{table}[htbp]
\centering
\small
\setlength{\tabcolsep}{4pt}
\begin{tabular}{ccccl}
\hline
HLS groupoid & A-T-menable & Metrically amenable
& Property $(\tau)$ & Reference \\
\hline
$G_{\mathcal V}(F)$ & Yes & Yes & No
& Ex.~\ref{ex:both properties}, Cor.~\ref{cor:tau obstruction} \\
$G_{\mathcal D}(F)$ & Yes & No & No
& Thm.~\ref{thm:reverse counterexample}, Cor.~\ref{cor:tau obstruction} \\
$G_{\mathcal M}(F)$ & No & Yes & No
& Thm.~\ref{thm:counterexample},
  Rem.~\ref{rem:no tau M} \\
$G_{\mathcal C}(F)$ & No & No & Yes
& Ex.~\ref{ex:F2 non+non} \\
\hline
\end{tabular}
\caption{The four combinations of a-T-menability
and metric amenability for HLS groupoids over
$F = \mathbb F_2$. Property $(\tau)$ refers to
$F$ relative to the corresponding
approximating sequence.}
\label{tab:four-cases}
\end{table}

\bibliographystyle{plain}
\bibliography{ref_HLS}

\end{document}